\documentclass{amsart}
\usepackage{amssymb}
\usepackage{amsmath}
\usepackage{color}
\usepackage{comment}
\usepackage{enumitem}

\newtheorem{prop}{Proposition}[section]
\newtheorem{lem}[prop]{Lemma}
\newtheorem{thm}[prop]{Theorem}
\newtheorem{cor}[prop]{Corollary}

\theoremstyle{remark}
\newtheorem{rem}[prop]{Remark}

\newtheorem{ex}[prop]{Example}

\theoremstyle{definition}
\newtheorem{defi}[prop]{Definition}

\newcommand{\cf}{{\mathcal F}}

\newcommand{\cc}{{\mathcal C}}
\newcommand{\cd}{{\mathcal D}}
\newcommand{\cq}{{\mathcal Q}}
\newcommand{\ck}{{\mathcal K}}
\newcommand{\cm}{{\mathcal M}}
\newcommand{\mf}{{\mathbb F}}

\newcommand{\qform}[1]{{\langle{#1}\rangle}} 
\newcommand{\pform}[1]{{\langle\!\langle{#1}\rangle\!\rangle}} 

\newcommand{\ra}{\rightarrow}

\DeclareMathOperator{\br}{Br}
\DeclareMathOperator{\Br}{Br}
\DeclareMathOperator{\coker}{coker}
\DeclareMathOperator{\End}{End}
\DeclareMathOperator{\id}{Id}
\DeclareMathOperator{\ind}{ind} 
\DeclareMathOperator{\SB}{SB}
\DeclareMathOperator{\Spec}{Spec}
\newcommand{\sq}[1]{#1^\times/#1^{\times 2}}

\newcommand{\sfA}{\mathsf{A}}
\newcommand{\sfB}{\mathsf{B}}
\newcommand{\sfC}{\mathsf{C}}
\newcommand{\sfD}{\mathsf{D}}

\newcommand{\PGO}{\mathrm{PGO}}
\newcommand{\PGL}{\mathrm{PGL}}
\DeclareMathOperator{\SL}{SL}
\DeclareMathOperator{\sign}{sign}
\DeclareMathOperator{\psim}{PSim^+}
\newcommand{\CH}{\mathrm{CH}}

\DeclareMathOperator{\ad}{ad} 
\DeclareMathOperator{\Ad}{Ad}
\newcommand{\ba}{\overline{\rule{2.5mm}{0mm}\rule{0mm}{4pt}}} 

\newcommand{\darkrad}{0.17}
\newcommand{\lrad}{0.4}

\title[Critical varieties for algebras with involution]{Critical varieties and motivic equivalence for algebras with involution}
\author{Charles de Clercq, Anne Qu\'eguiner-Mathieu and Maksim Zhykhovich}
\date{\today}

\begin{document}

\begin{abstract}
Motivic equivalence for algebraic groups was recently introduced in~\cite{dC}, where a characterization of motivic equivalent groups in terms of higher Tits indexes is given. As a consequence, if the quadrics associated to two quadratic forms have the same Chow motives with coefficients in $\mf_2$, this remains true for any two projective homogeneous varieties of the same type under the orthogonal groups of those two quadratic forms. Our main result extends this to all groups of classical type, and to some exceptional groups, introducing a notion of critical variety. On the way, we prove that motivic equivalence of the automorphism groups of two involutions can be checked after extending scalars to some index reduction field, which depends on the type of the involutions. In addition, we describe conditions on the base field which guarantee that motivic equivalent involutions actually are isomorphic, extending a  result of Hoffmann on quadratic forms. 
\end{abstract}

\maketitle
\section{Introduction and notations}
\label{introduction.section}

Andr\'e Weil proved in the 60's that classical algebraic groups can be described in terms of automorphism groups of central simple algebras with involution~\cite{W}. This is a key tool for important results both on algebraic groups, and on related algebraic objects, such as algebras with involution, or quadratic and hermitian forms. As a major example, we mention Bayer and Parimala's proof of Serre's conjecture II~\cite{BPInventiones}  and of the Hasse principle conjecture II~\cite{BPAnnals}, which provide a  description of the first cohomology group of a simply connected classical group when the base field has cohomological dimension $\leq 2$, or virtual cohomological dimension $\leq 2$, respectively. Besides Galois cohomology, another important tool must be added to this picture, namely the study of projective homogeneous varieties under a semisimple algebraic group. This led, for instance, to the so-called index reduction formulae, due to Merkurjev, Panin and Wadsworth~\cite{MPW1}~\cite{MPW2}.

In this paper, we use a scalar extension to an index reduction field of the underlying algebra to study the Chow motives of projective homogeneous varieties under an algebraic group of classical type, and the relation to cohomological invariants of associated algebraic objects. 
More precisely, recall the notion of motivic equivalence modulo $2$ for algebraic groups introduced by the first named author in~\cite{dC} : two algebraic groups which are inner twisted forms of the same quasi-split group are called motivic equivalent modulo $2$ if the projective homogeneous varieties of the same type have the same Chow motives with coefficients in $\mf_2$. A characterization of motivic equivalent groups is also provided in the same paper, in terms of higher Tits indexes, which consist of the collection of Tits indexes of the given group over all field extensions of the base field. 
As a consequence of this result, one may check that if the quadrics associated to two quadratic forms have the same Chow motives with coefficients in $\mf_2$, this remains true for all projective homogeneous varieties under the orthogonal groups of those two quadratic forms. 
We extend this result to algebraic groups of classical type. More precisely, we introduce a projective homogeneous variety, called a critical variety, and depending only on the type of the group, which plays the same role as a quadric for an orthogonal group, that is which is enough to track down motivic equivalence mod $2$ (see Definition~\ref{critical.def} and Theorem~\ref{main.thm}). Hence, given two algebras with involution of the same type, if the corresponding critical varieties (see Definition~\ref{descriptioncritical.def}) have the same Chow motive with coefficients in $\mf_2$, then this remains true for all projective homogeneous varieties under the automorphism groups of both involutions. As a consequence, we also introduce a notion of motivic equivalence for involutions, and we prove a generalization of Vishik's criterion for motivic equivalence of quadrics (see Corollary~\ref{Vishik.cor}).

It appears that motivic equivalence for two involutions defined on the same algebra can be checked after applying an index reduction process to the underlying central simple algebra. As a consequence, the low-degree cohomological invariants of two motivic equivalent involutions coincide. This can be used to explore the relation between motivic equivalence and isomorphism. Izhboldin proved in~\cite{I} that motivic equivalent odd dimensional quadratic forms are similar, hence the corresponding orthogonal groups are isomorphic. He also gave examples of even dimensional motivic equivalent and non-similar quadratic forms, for which the corresponding adjoint involutions are motivic equivalent and non-isomorphic. In section~\ref{iso.section}, we provide conditions on the base field under which motivic equivalent involutions are isomorphic, extending a result previously proved by Hoffmann in~\cite{Hoffmann} in the quadratic form case. 

The notion of critical variety clearly extends to exceptional groups, where one may also have to consider Chow motives mod $p$ for some odd prime numbers, depending on the type of the group. In the final section~\S\ref{exceptional.section}, we provide examples of critical varieties for motivic equivalence mod $p$ for some types of exceptional groups and some prime numbers $p$. 

As a major tool in this paper, we use the  anisotropy results of Karpenko and Karpenko-Zhykhovich~(\cite{K},~\cite{KZ}), which can be rephrased in terms of $2$-Witt indices of involutions, see Lemma~\ref{Wittindex.lemma}, as well as the main result in~\cite{dC}, which is rephrased in Lemma~\ref{Charles.lem}. \\

\noindent {\bf Acknowledgements.} We are grateful to Skip Garibaldi for his careful reading of a preliminary version of the paper, and his useful comments and suggestions. The third named author also thanks Nikita Semenov for helpful discussions.

\subsection*{Notations}
Throughout this paper, we work over a base field of characteristic different from $2$.
Given an involution $\sigma$ on a central simple algebra $A$ of degree $n$ over a field $K$, we denote by $F$ the subfield of $K$ fixed by $\sigma$, hence $K=F$ if $\sigma$ is orthogonal or symplectic, and $K=F(\sqrt\delta)$ is a quadratic field extension of $F$ if $\sigma$ is unitary. We also allow $A$ to be a direct product of two central simple algebras over $F$, so that its center is the quadratic \'etale algebra $K=F\times F$, endowed with an involution which acts on $F\times F$ as the unique non-trivial $F$-automorphism. If so, $(A,\sigma)$ is isomorphic to the algebra $E\times E^{\mathrm{op}}$ for some $F$ central simple algebra $E$, endowed with the exchange involution $\varepsilon$ defined by $\varepsilon(x,y^{\mathrm{op}})=(y,x^{\mathrm{op}})$, see~\cite[(2.14)]{KMRT}. We then call degree and index of $A$ the degree and the Schur index of $E$, which coincides with the degree and the Schur index of $E^{\mathrm{op}}$, respectively. 
In all cases, {$A$ admits no non trivial $\sigma$-stable two-sided ideal}, and we will say that $(A,\sigma)$ is a central simple algebra with involution over $F$, or an algebra with involution over $F$, for short.
 With this convention,
for any field extension $L/F$, the pair $(A_L,\sigma_L)$ defined by $A_L=A\otimes_F L$ and $\sigma_L=\sigma\otimes \id$ is an algebra with involution over $L$.

We refer the reader to~\cite{KMRT} for more details on algebras with involution and on the corresponding algebraic groups. In  particular, we let $\psim(A,\sigma)$ be the connected component of the identity in the automorphism group of $(A,\sigma)$. It is a semisimple $F$-algebraic group, which is of  type $\sfA$, $\sfB$, $\sfC$ or $\sfD$ depending on the degree of the algebra $A$ and on the type of the involution $\sigma$. Conversely, by~\cite{W} (see also~\cite[\S 26]{KMRT}), any algebraic group of classical type over $F$ is isogeneous to $\psim(A,\sigma)$ for some algebra with involution $(A,\sigma)$ over $F$. 
Note that given two algebras with involution $(A,\sigma)$ and $(B,\tau)$, the corresponding algebraic groups $\psim(A,\sigma)$ and $\psim(B,\tau)$ are inner twisted forms of the same quasi-split group if and only if the involutions have the same type, the algebras have the same degree, and in addition, the algebras have isomorphic center in unitary type, and the involutions have the same discrimininant in orthogonal type. 

All the quadratic forms we consider are supposed to be non degenerate. The invariants are as defined in~\cite{Sch},~\cite{Lam}. In particular, the discriminant $d_\varphi\in\sq F$ of a quadratic form $\varphi$ over $F$ is a signed discriminant, and the Clifford invariant $c(\varphi)\in\br(F)$, {called the Witt invariant in Lam's book, see ~\cite[Chap.3, (3.12)]{Lam}},  is the Brauer class of the full Clifford algebra $\cc(\varphi)$ if $\varphi$ has even dimension, and of its even part $\cc_0(\varphi)$ if $\varphi$ has odd dimension. The Witt index of $q$ is denoted by $i_w(q)$.

Given an algebra with involution $(A,\sigma)$ over $F$, and assuming it has non split center in the unitary case, we call Witt-index of $\sigma$, and we denote by $i_w(\sigma)$, the reduced dimension of the maximal totally isotropic right ideals in $(A,\sigma)$. Therefore, $i_w(\sigma)$ is the largest element in the index of $(A,\sigma)$ as defined in~\cite[\S 6.A]{KMRT}. In particular, if $(A,\sigma)$ is hyperbolic, $A$ has even degree and the Witt index of $\sigma$ equals $\frac12\deg(A)$. If $\sigma$ is unitary, and $A$ has center $F\times F$, we have $(A,\sigma)\simeq (E\times E^{\mathrm {op}},\varepsilon)$, and the involution is hyperbolic~\cite[(6.8)]{KMRT}, even though the algebra $E$ might be of odd degree. Therefore, we define $i_w(\varepsilon)$ to be the integral part of $\frac 12\deg(E)$. Note that, as opposed to what happens in other cases, the Witt index of $\sigma$ coincides with the maximal element in the index of $(A,\sigma)$ if and only if $E$ has even degree and its Schur index divides $\frac 12\deg(E)$ (see~\cite[p. 73]{KMRT}). In particular, if $E$ is division, then even though $\sigma$ is hyperbolic and has non-trivial Witt index, the index of $(A,\sigma)$ is $\{0\}$ and its automorphism group is anisotropic. 

The Witt $2$-index of $(A,\sigma)$,  denoted by $i_{w,2}(\sigma)$, is defined as the maximal value of $i_w(\sigma_{F'})$, where $F'$ runs over all odd degree field extensions of $F$. As explained in~\cite{dCG}, the Tits-index (respectively the $2$-Tits index) of $\psim(A,\sigma)$  is uniquely determined by the Schur index of the algebra $A$ and by $i_w(\sigma)$ (respectively {by the $2$-primary part of the Schur index of $A$} and $i_{w,2}(\sigma)$).

{If $A$ has even degree}, and $\sigma$ has orthogonal type, we let $d_\sigma\in F^\times/F^{\times 2}$ be the discriminant of $\sigma$~\cite[(7.2)]{KMRT}, and $\cc(\sigma)$ be the Clifford algebra of $(A,\sigma)$~\cite[(8.7)]{KMRT}. We denote by $K_\sigma/F$ the quadratic \'etale extension associated to $d_\sigma$, which also is the center of $\cc(\sigma)$~\cite[(8.10)]{KMRT}.
If $d_\sigma=1\in F^\times/F^{\times 2}$, then $\cc(\sigma)$ is a direct product of two central simple algebras over $F$, denoted by $\cc_+(\sigma)$ and $\cc_-(\sigma)$. By the so-called fundamental relations~\cite{KMRT}, we have $[\cc_+(\sigma)]-[\cc_-(\sigma)]=[A]\in\br(F)$. The Clifford invariant of $\sigma$ is \[c(\sigma)=[\cc_+(\sigma)]=[\cc_-(\sigma)]\in\br(F)/\qform{[A]}.\]

If $A$ has even degree, and $\sigma$ is unitary, we let $\cd(\sigma)$ be the discriminant algebra of $\sigma$, which is a central simple algebra over $F$~\cite[(10.28)]{KMRT}. 

Assume now that the field $F$ is formally real, and denote by $X_F$ the space of orderings of $F$. For all $P\in X_F$, we let $\sign_P(\sigma)$ be the signature of $\sigma$ at the ordering $P$ (see~\cite[(11.10)(11.25)]{KMRT} and~\cite[\S 4]{AU}). By definition, $\sign_P(\sigma)$ is a positive integer, whose square is equal to the signature at $P$ of the trace form of $(A,\sigma)$. If $\sigma$ is $K/F$-unitary, with $K=F(\sqrt \alpha)$, then 
\begin{equation}
\label{sign.eq}
\sign_P(\sigma)=0\mbox{ for all }P\in X_F\mbox{ such that }\alpha>_P 0.
\end{equation}
Indeed, over a real closure $F_P$ of $F$ at such an ordering, we have $K\otimes_F F_P=F_P\times F_P$, so that all hermitian forms with values in $(K,\iota)$ become hyperbolic over $F_P$ (see~\cite[3.1(d)]{AU}). This applies in particular to the trace form of $(A,\sigma)$. Therefore, as noticed in~\cite{AU}, one may extend the definition given in~\cite[\S\,11]{KMRT} and~\cite{Q:signature} to the orderings of $F$ that do not extend to $K$, with the convention above.

The motives considered in this paper are Chow motives with coefficients in a ring $\Lambda$, see for instance~\cite{EKM}. In particular, if $X$ is a variety over $F$, we denote in the sequel by $\CH(X;\Lambda)$ and  $M(X;\Lambda)$ the Chow group and motive associated to $X$ with coefficient ring $\Lambda$. We may sometimes omit the coefficient ring when it is clear from the context.

\section{Involutions, low index algebras and index reduction fields}

Assume $A=\End_F(V)$ is a split algebra. Any orthogonal involution on $A$ is adjoint to some quadratic form $q:\,V\ra F$. Conversely, two quadratic forms give rise to the same involution if and only if they are isomorphic up to a scalar factor.
Hence, the study of orthogonal involutions in this case boils down to quadratic form theory, or more precisely to the study of quadratic forms up to similarities. In particular, all invariants of an orthogonal involution on a split algebra can be computed in terms of invariants of quadratic forms. As we now proceed to explain, this is also true for symplectic involutions on an index at most $2$ algebra and unitary involutions on a split algebra. 
Moreover, most invariants of involutions can be computed after an index reduction process, so that they can be expressed in terms of some quadratic form invariants.

\subsection{Invariants of involutions}
\label{lowindex.section}

Let us first recall precisely the situation in the split orthogonal case. For any quadratic form $q:\,V\ra F$, we denote by $\Ad_q=(\End_F(V),\ad_q)$ the associated algebra with involution. It is well known that invariants of orthogonal involutions extend invariants of quadratic forms. More precisely, we have :
\begin{lem}
\label{splitorthogonal.lemma}
Let $q$ be an even-dimensional quadratic form over $F$, and let $\sigma=\ad_q$ be the corresponding adjoint involution.
\begin{enumerate}[label=(\alph*)]
\item $i_w(\sigma){=i_{w,2}(\sigma)}=i_w(q)$;
\item $d_\sigma=d_q\in\sq F$;
\item $\cc(\sigma)$ and $\cc_0(q)$ are canonically isomorphic;
\item $c(\sigma_{K})=c(q_{K})$, where $K$ is the discriminant quadratic extension;
\item For all $P\in X_F$, $\sign_P(\sigma)=|\sign_P(q)|$.
\end{enumerate}
\end{lem}
\begin{rem}
The Clifford algebra of an involution, denoted here by $\cc(\sigma)$, corresponds in the split case to the even part $\cc_0(q)$ of the Clifford algebra of $q$. The full Clifford algebra $\cc(q)$ of the quadratic form $q$ is not an invariant up to similarities; therefore, it does not give rise to an invariant of $\sigma$. {Likewise, $\sign(q)$ is not an invariant up to similarities, hence the absolute value is required in (e)}. 
\end{rem}
\begin{proof}
In the situation considered here, the underlying division algebra is $D=F$, therefore by~\cite[p.73]{KMRT}, the maximal element of the index of $(A,\sigma)$ is the Witt-index $i_w(q)$, hence $i_w(\sigma)=i_w(q)$. {Moreover, for any odd degree field extension $F'/F$, by Springer's theorem for quadratic forms (see e.g.~\cite[Chap. VII, Thm. 2.7]{Lam}), we have 
$$i_w(\sigma_{F'})=i_w(q_{F'})=i_w(q)=i_w(\sigma).$$ Hence $i_{w,2}(\sigma)=i_w(\sigma)$ and (a) is proved.}
Assertions (b) and (c) are given in~\cite[(7.3)(3) and (8.8)]{KMRT}.
Since $K$ is the discriminant quadratic extension, in order to prove (d), it is enough to check that $c(\sigma)=c(q)$ when the discriminants are trivial. Under this assumption, we have $\cc(q)\simeq M_2(C)$ for some central simple algebra $C$ over $F$, and $\cc_0(q)\simeq C\times C$ {(see e.g.~\cite[Chap. V, Thm 2.5(3)]{Lam})}. Therefore $[\cc_+(\sigma)]=[\cc_-(\sigma)]=[C]=[\cc(q)]\in\br F$, which gives the required equality.
Finally, if $P$ is any ordering of $F$, we have $\sign_P(\sigma)=|\sign_P(q)|$ by~\cite[(11.10) and (11.7)]{KMRT}.
\end{proof}

Let us assume now that either $\sigma$ is unitary and $A=M_n(K)$, or $\sigma$ is symplectic and $A=M_m(H)$ for some quaternion algebra $H$ over $F$. Let $n$ be the degree of $A$, that is $n=2m$ if $A=M_m(H)$. In both cases, there exists a $2n$-dimensional quadratic form $q_\sigma$ over $F$, which is unique up to a scalar factor, and whose invariants are related to the invariants of $\sigma$, as we now proceed to explain.  

A unitary involution $\sigma$ on the split algebra $M_n(K)$ is adjoint to a rank $n$ hermitian form, denoted by $h_\sigma$, with values in $(K,\ba)$, where $\ba$ denotes the non-trivial $F$-automorphism of the quadratic extension $K/F$. Similarly, a symplectic involution $\sigma$ of the algebra $M_m(H)$ is adjoint to a rank $m$ hermitian form, still denoted by $h_\sigma$, with values in $(H,\ba)$, where $\ba$ stands for the canonical involution of $H$. Let $q_\sigma$ be the trace form of $h_\sigma$, that is the quadratic form defined on $K^n$ (respectively $H^m$), now viewed as an $F$ vector space, by $q_\sigma(x)=h_\sigma(x,x)$. {It has dimension $2n$ in the unitary case, and $4m=2n$ in the symplectic case. Moreover,} as explained in~\cite[Chap.10]{Sch}, the hermitian form $h_\sigma$ is uniquely determined by its trace form $q_\sigma$. Invariants of $\sigma$ and invariants of $q_\sigma$ are related as follows:
\begin{prop}
\label{lowindex.prop} Let $\sigma$ be either a unitary involution of a split algebra $M_n(K)$ or a symplectic involution of an algebra $M_m(H)$ of index at most $2$, and consider as above a trace form $q_\sigma$ of the underlying hermitian form. This form is well defined up to a scalar factor, and its invariants relate to those of $\sigma$ as follows :
\begin{enumerate}[label=(\alph*)]
\item $i_w(\sigma){=i_{w,2}(\sigma)}$ coincides with the integral part of $\frac 12 i_w(q_\sigma)$;
\item If $A$ has even degree and $\sigma$ is unitary, \[d_{q_\sigma}=1\in\sq F\mbox{ and }[\cd(\sigma)]=c(q_\sigma)\in\br(F);\]
\item For all $P\in X_F$, $\sign_P(\sigma)=\frac 12|\sign_P(q_\sigma)|$.
\end{enumerate}
\end{prop}

\begin{proof}
By the same argument as in the orthogonal case, based on Springer's theorem for quadratic forms, to prove (a), it is enough to prove the equality $i_w(\sigma)=\frac12i_w(q_\sigma)$. 
Pick an arbitrary diagonalisation $h_\sigma=\qform{\alpha_1,\dots,\alpha_r}$ of $h_\sigma$, where $r=n$ in the unitary case, and $r=m=\frac n2$ in the symplectic case. The elements $\alpha_i$ are symmetric elements of $(K,\ba)$ or $(H,\ba)$ depending on the type of $\sigma$; hence in both cases, they belong to $F$. An easy computation now shows that
\[q_\sigma\simeq\left\{\begin{array}{ll}
\qform{1,-\delta}\otimes\qform{\alpha_1,\dots,\alpha_n}&\mbox{ if }\sigma\mbox{ is unitary};\\
n_H\otimes \qform{\alpha_1,\dots,\alpha_m}&\mbox{ if }\sigma\mbox{ is symplectic},\\
\end{array}\right.\]
where $n_H$ denotes the norm form of the quaternion algebra $H$. Since $h_\sigma$ is uniquely defined up to a scalar factor, the same holds for $q_\sigma$.

If $K=F\times F$, that is $\delta=1$, or $H$ is split, then the involution $\sigma$ is hyperbolic and $i_w(\sigma)$ is the integral part of $\frac12\deg(A)$.
On the other hand, under those assumptions, $\qform{1,-\delta}$ and $n_H$ are hyperbolic, hence $q_\sigma$ also is, and $i_w(q_\sigma)=\frac12 \dim(q_\sigma)=\deg(A)$. Therefore (a) holds in this case. In addition, if $\sigma$ is symplectic, then $A$ has even degree so that $i_w(\sigma)=\frac 12 i_w(q_\sigma)$. 
Moreover, for any ordering $P$ of $F$, we have $1>_P0$, hence $\sign_P(\sigma)=0$ by~(\ref{sign.eq}) in the unitary case and~\cite[(11.7)]{KMRT} in the symplectic case. Therefore, since $q_\sigma$ is hyperbolic, (c) also holds in this case.

Assume now that $K$ is a field or $H$ is division, depending on the type of the involution.
Combining the explicit description of $q_\sigma$ above with~\cite[Chap.10, Thm1.1 and 1.7]{Sch}, one may observe that any Witt decomposition of the hermitian form $h_\sigma$ gives rise to a Witt decomposition of $q_\sigma$, so that
\[i_w(q_\sigma)=\left\{\begin{array}{ll}
2 i_w(h_\sigma)&\mbox{ if }\sigma\mbox{ is unitary};\\
4 i_w(h_\sigma)&\mbox{ if }\sigma\mbox{ is symplectic}.\\
\end{array}\right.\]
On the other hand, from the description of the index of an algebra with involution given in~\cite[p.73]{KMRT}, $i_w(\sigma)$ is equal to $i_w(h_\sigma)$ in the unitary case, and $2i_w(h_\sigma)$ in the symplectic case, since the index of the underlying algebra is $1$ or $2$, accordingly. Therefore (a) holds.

In order to prove (c), let us consider an ordering $P$ of the field $F$, and pick a real closure $F_P$ of $F$ at this ordering. 
Assume first that $\sigma$ is unitary. 
If $\delta>_P0$, we have $\sign_P(\sigma)=0$ by~(\ref{sign.eq}), hence (1) holds since $q_\sigma$ is hyperbolic over $F_P$. If $\delta<_P0$, the diagonalisation above gives $\sign_P(\sigma)=|\sign_P(h_\sigma)|=\frac12|\sign_P(q_\sigma)|$ as required. Assume now $\sigma$ is symplectic. The quaternion algebra $H$ is split over $F_P$ if and only if its norm form $n_H$ has signature $0$. When these conditions hold, we have by~\cite[(11.11)(2)(b)]{KMRT} $\sign_P(\sigma)=0=\sign_P(q_\sigma)$. Otherwise, $n_H$ is positive definite, and $\sign_P(\sigma)=2|\sign_P(h_\sigma)|=\frac12|\sign_P(q_\sigma)|$. Hence, (c) is now proved.

It only remains to prove assertion (b). Assume $A$ has even degree, $n=2m$, so that the discriminant algebra $\cd(\sigma)$ is well defined. The diagonalisation of $q_\sigma$ given above shows that $d(q_\sigma)=(-\delta)^{2m}=1\in\sq F$. Moreover, by~\cite[(10.35)]{KMRT}, $\cd(\sigma)$ is Brauer equivalent to $(\delta,(-1)^m\alpha_1\dots\alpha_n)$, which is nothing but the Clifford invariant of $q_\sigma$. This concludes the proof of the proposition.
\end{proof}

\subsection{Index reduction function fields}
\label{functionfield.section}
Let $A$ be a central simple algebra of arbitrary index.
The discriminant and the Clifford invariant in the even degree orthogonal case, the discriminant algebra in the even degree unitary case, and the signatures of two involutions of the algebra $A$ can be compared after a scalar extension to  a well chosen function field, which reduces the index of $A$.
More precisely, given an algebra with involution $(A,\sigma)$ over $F$ of type $t$, where $t=o$ (respectively $s$,$u$) stands for orthogonal (respectively symplectic, unitary), we consider the field  $\cf_{A,t}$ and the quadratic form $\cq_\sigma$ over $\cf_{A,t}$ defined as follows.
We let $\cf_{A,o}$ be the function field of the Severi-Brauer variety of $A$ in the orthogonal case, $\cf_{A,s}$ be the function field of the generalized Severi-Brauer variety $\SB_2(A)$ of right ideals of reduced dimension $2$ in the symplectic case, $\cf_{A,u}$ be the function field of the $K/F$-Weil transfer of the Severi-Brauer variety of $A$ in the unitary case with non split center, and $\cf_{A,u}$ be the function field of the Severi-Brauer variety of $E$ if $A\simeq E\times E^{\mathrm{op}}$ with the exchange involution. Hence, $\cf_{A,t}$ is a field extension of $F$ in all cases, and the algebras $A\otimes_F{\cf_{A,o}}$ and $A\otimes_F {\cf_{A,u}}$ are split, while $A\otimes_F{\cf_{A,s}}$ has index at most $2$, and is split if and only if $A$ already is.
Therefore, by \S\,\ref{lowindex.section}, there exists a quadratic form $\cq_\sigma$, defined over $\cf_{A,t}$, unique up to a scalar factor, and of dimension $n$ in the orthogonal case, and $2n$ in the symplectic and unitary cases, which determines the involution $\sigma_{\cf_{A,t}}$ and its invariants. Using the properties of the field $\cf_{A,t}$, we get the following :
\begin{prop}
\label{functionfield.prop}
Let $(A,\sigma,\tau)$ be a central simple algebra with two involutions of the same type $t$ over $F$. Let $\cf_{A,t}$ be the index reduction field defined above, and denote by $\cq_\sigma$ and $\cq_\tau$ the quadratic forms over $\cf_{A,t}$ respectively associated to $\sigma_{\cf_{A,t}}$ and $\tau_{\cf_{A,t}}$, as in \S\,\ref{lowindex.section}.

(a) Assume $A$ has even degree and $\cq_\sigma$ and $\cq_\tau$ have the same discriminant. We denote by $\ck/\cf_{A,t}$ the corresponding quadratic \'etale algebra, and assume in addition that $c((\cq_\sigma)_\ck)=c((\cq_\tau)_\ck)\in\br(\ck)$. Then, the following hold : 
\begin{enumerate}
\item[(i)] If $\sigma$ and $\tau$ are orthogonal, then \[d_\sigma=d_\tau\in\sq F\]\[\mbox{and }c(\sigma_K)=c(\tau_K)\in\br(K)/\qform{[A_K]},\] where $K=F[X]/(X^2-d_\sigma)$ is the discriminant quadratic extension.
\item[(ii)] If $\sigma$ and $\tau$ are unitary, then \[[\cd(\sigma)]=[\cd(\tau)]\in\br(F).\]
\end{enumerate}

(b) Assume that for all $Q\in X_{\cf_{A,t}}$ we have $|\sign_Q(\cq_\sigma)|=|\sign_Q(\cq_\tau)|$. Then for all $P\in X_F$ \[\sign_P(\sigma)=\sign_P(\tau).\]
\end{prop}

\begin{proof}
Let us first prove (a) in the orthogonal case. By lemma~\ref{splitorthogonal.lemma}, the assumptions on $\cq_\sigma$ and $\cq_\tau$ guarantee that $\sigma_{\cf_{A,o}}$ and $\tau_{\cf_{A,o}}$ have the same discriminant, corresponding to the quadratic \'etale algebra $\ck/\cf_{A,o}$, and that $c(\sigma_\ck)=c(\tau_\ck)$. Since the discriminant is a functorial invariant, and $F$ is quadratically closed in $\cf_{A,o}$, it follows that $\sigma$ and $\tau$ have the same discriminant. Denote by $K/F$ the corresponding quadratic \'etale algebra.
If $d_\sigma=d_\tau=1\in\sq F$, so that $K=F\times F$ and $\ck=\cf_{A,o}\times \cf_{A,o}$, then $c(\sigma_\ck)=(c(\sigma_{\cf_{A,o}}),c(\sigma_{\cf_{A,o}}))\in\br(\cf_{A,o})\times\br(\cf_{A,o})$, and similarly for $\tau_\ck$. Therefore the assumptions gives \[c(\sigma_{\cf_{A,o}})=c(\tau_{\cf_{A,o}})\in \br(\cf_{A,o}).\] By~\cite[Cor.2.7]{MT}, the kernel of the restriction map $\br(F)\ra\br(\cf_{A,o})$ is the subgroup generated by $[A]$, therefore $\br(F)/\qform{[A]}\ra\br(\cf_{A,o})$ is injective. Hence, we get \[c(\sigma)=c(\tau)\in \br(F)/\qform{[A]}\] which implies (i). Assume now that $K$ is a field.
The quadratic algebra $\ck$ is the compositum of $K$ and $\cf_{A,o}$, or equivalently, the function field over $K$ of the Severi-Brauer variety $A_K$. Hence, by the same argument as above, the map $\br K/\qform{[A_K]}\ra\br\ck$ is injective. Therefore, again in this case, \[c(\sigma_\ck)=c(\tau_\ck)\in\br(\ck)\] implies \[c(\sigma_K)=c(\tau_K)\in \br K/\qform{[A_K]}.\]

Assume now that the assumptions of (a) hold and that $\sigma$ and $\tau$ are unitary. By Proposition~\ref{lowindex.prop}, $\cq_\sigma$ and $\cq_\tau$ have trivial discriminant. Hence $\ck$ is the split quadratic \'etale algebra $\cf_{A,u}\times\cf_{A,u}$ and 
$c((\cq_\sigma)_\ck)=c((\cq_\tau)_\ck)\in\br(\ck)$ implies
\[c(\cq_\sigma)=c(\cq_\tau)\in\br(\cf_{A,u}).\] 
By Proposition~\ref{lowindex.prop}, we get that $[\cd(\sigma_{\cf_{A,u}})]=[\cd(\tau_{\cf_{A,u}})]\in\br(\cf_{A,u})$. Since $A$ admits unitary $K/F$ involutions, its corestriction is split; therefore, by~\cite[Cor.2.7, Cor.2.12]{MT}, the restriction map $\br(F)\ra\br(\cf_{A,u})$ is injective (see also~\cite[Proof of (10.36)]{KMRT}). Hence, we get \[[\cd(\sigma)]=[\cd(\tau)]\in\br(F)\] as required.

As for the signatures, assume $|\sign_Q(\cq_\sigma)|=|\sign_Q(\cq_\tau)|$ for all $Q\in X_{\cf_{A,t}}$ and consider an ordering $P\in X_F$. Pick a real closure $F_P$ of $F$ at the ordering $P$. If $\sigma$ and $\tau$ are orthogonal, then for all ordering $P$ such that $A_{F_P}$ is not split, we have $\sign_P(\sigma)=0=\sign_P(\tau)$ by~\cite[(11.11)]{KMRT}. Consider now $P\in X_F$ such that $A_{F_P}$ is split. The compositum of $\cf_{A,o}$ and $F_P$, which is the function field over $F_P$ of the Severi-Brauer variety of $A_{F_P}$ is a purely transcendental extension of $F_P$. Therefore, the ordering $P$ extends to this field, and by restriction, there exists an ordering $Q\in\cf_{A,o}$ which coincides with $P$ over $F$. Therefore, $\sign_P(\sigma)=\sign_Q(\sigma_{\cf_{A,o}})=|\sign_Q(\cq_\sigma)|$, where the last equality follows from Lemma~\ref{splitorthogonal.lemma},  and similarly for $\tau$. Hence $\sigma$ and $\tau$ do have the same signature at $P$ for all $P\in X_F$.
The argument is similar in the symplectic and unitary cases, and uses Proposition~\ref{lowindex.prop}. If $\sigma$ and $\tau$ are symplectic, they both have trivial signature at $P$ for all $P\in X_F$ such that $A_{F_P}$ is split by~\cite[(11.11)]{KMRT}. Otherwise, $A_{F_P}$ is Brauer equivalent to $(-1,-1)_{F_P}$, which is the only non split division algebra over $F_P$. Therefore, $\SB_2(A)$ has a rational point over $F_P$, its function field is purely transcendental, and the same argument as above applies. Finally, if $\sigma$ and $\tau$ are unitary, they both have trivial signature at any ordering $P\in X_F$ such that $\delta>_P0$ by~(\ref{sign.eq}). Consider now an ordering $P$ such that $\delta<_P0$. The compositum of $F_P$ and $K$ is the unique non trivial quadratic field extension of $F_P$, that is an algebraically closed field. Hence, $A\otimes_F F_P$ is split, so the Weil transfer of $\SB(A)$ has a rational point over $F$ and again the same argument concludes the proof.
\end{proof}

As opposed to what happens for the invariants considered in the previous proposition, it is not known whether the Witt indices can be compared after scalar extension to $\cf_{A,t}$. Nevertheless, we have the following, which is precisely what we need in the sequel, and which can be thought of as a reformulation of Thm 1 and Thm A in~\cite{K} and Thm 6.1 in~\cite{KZ}:
\begin{lem}
\label{Wittindex.lemma}
For any involution $\sigma$ of type $t$ on the algebra $A$, we have
$$i_{w,2}(\sigma)=i_{w,2}(\sigma_{\cf_{A,t}})=i_w(\sigma_{\cf_{A,t}}),$$
where $\cf_{A,t}$ is the function field defined above. Moreover, this index coincide with $i_w(\cq_\sigma)$ if $\sigma$ is orthogonal, and $\frac 12 i_w(\cq_\sigma)$ if it is symplectic or unitary.
\end{lem}
\begin{proof}
{By definition of $\cf_{A,t}$, Lemma~\ref{splitorthogonal.lemma} or Proposition~\ref{lowindex.prop} apply to the involution $\sigma_{\cf_{A,t}}$, depending on its type. Therefore we have $i_{w,2}(\sigma_{\cf_{A,t}})=i_w(\sigma_{\cf_{A,t}})$, and this index coincides with $i_w(\cq_\sigma)$, $\frac 12 i_w(\cq_\sigma)$, or the integral part of $\frac 12 i_w(\cq_\sigma)$, depending on the type of the involution. 
Since the $2$-Witt index can only increase under extension to $\cf_{A,t}$,} we get
$$i_{w,2}(\sigma)\leq i_{w,2}(\sigma_{\cf_{A,t}})=i_w(\sigma_{\cf_{A,t}}).$$
The converse inequality is clear in the unitary case if the center is $F\times F$. 
In all other cases, by an easy induction argument, it follows from Thm1 and Thm A in~\cite{K} and Thm 6.1 in~\cite{KZ} which precisely state that if $\sigma$ is isotropic over $\cf_{A,t}$, then it also is isotropic over some odd degree field extension of $F$ (see also~\cite[Cor 5.6]{BQ}).
\end{proof}

\section{Motivic equivalence and critical varieties}

The main result of this section asserts that we may associate to any absolutely almost simple algebraic group of classical type a projective homogeneous variety, which we call a critical variety, and whose motive modulo $2$ encodes the motivic equivalence class of the group. 

More precisely, let $G$ be semi-simple algebraic group over $F$; it is an inner twisted form of a given quasi-split group $G_0$. We choose a Borel subgroup of $G_0$ containing a maximal torus $T_0$, and we denote by $\Delta$ the corresponding set of simple roots. We recall from~\cite[\S VI]{dC} the definition of the standard motive of $G$ of type $\Theta$ with coefficients in $\Lambda$, denoted by $M_{\Theta,G}$, where $\Theta$ is any subset of $\Delta$. If $\Theta$ is invariant under the $\ast$-action, $M_{\Theta,G}$ is the motive, with coefficients in $\Lambda$, of the variety $X_{\Theta, G}$ of parabolic subgroups of $G$ of type $\Theta$, as defined in~\cite{BT}.
In general, $M_{\Theta, G}$ is the motive of the corestriction from $F_\Theta$ to $F$ of the variety $X_{\Theta, G_{F_\Theta}}$, where $F_\Theta/F$ is a minimal field extension over which $\Theta$ becomes invariant under the $\ast$-action.
Note that there are two opposite conventions for the parabolic subgroup of type $\Theta$ in the literature; in this paper, a Borel subgroup has type $\Delta$.

Assume now that $G$ and $G'$ are inner twisted forms of the same quasi-split group $G_0$. Recall from~\cite[Def 1]{dC} and~\cite[\S I.3]{dCG} that they are called motivic equivalent with coefficients in $\Lambda$ if there exists an isomorphism $f$ between the respective Dynkin diagrams of $G$ and $G'$, commuting with the Galois actions, and such that the standard motives $M_{\Theta, G}$ and $M_{f(\Theta), G'}$ with coefficients in $\Lambda$ are isomorphic for all $\Theta\subset\Delta$. When this holds for $\Lambda=\mf_p$, $G$ and $G'$ are called motivic equivalent modulo $p$.
The following definition formalizes the notion of critical variety.
\begin{defi}
\label{critical.def}
Let $G$ be an algebraic group, and denote by $G_0$ the corresponding quasi-split group. A twisted flag $G$-variety $X_{\Theta,G}$ is called critical for $G$ modulo $p$ if, for all group $G'$ which is an inner twisted form of the same $G_0$, the following assertions are equivalent :
\begin{enumerate}
\item There exists an isomorphism $f$ between the respective Dynkin diagrams of $G$ and $G'$, commmuting with the Galois action, and such that the standard motives $M_{\Theta,G}$ and $M_{f(\Theta),G'}$ with coefficients in $\mf_p$ are isomorphic; 
\item The algebraic groups $G$ and $G'$ are motivic equivalent modulo $p$.
\end{enumerate}
\end{defi}

\begin{rem}
 By definition of motivic equivalence, condition (2) always implies condition (1). Therefore, $X_{\Theta,G}$ is critical for $G$ if (1) implies (2). 
 Also, note that it may happen, for instance for groups of type $\sfD_n$, that condition (2) holds only for a specific choice of the isomorphism $f$, while condition (1) holds for two different choices of $f$.  
 
 Since the twisted flag varieties under a group $G$ only depend on the isogeny class of $G$, we only consider adjoint groups in the sequel. 
\end{rem}

\begin{ex}
\label{splitodd.ex}
Let $q$ be an odd-dimensional quadratic form; we claim that the quadric $X_q=X_{\{1\},G}$ is critical for the group $G=\PGO^+(q)$ modulo $2$. Indeed, any adjoint group $G'$ of the same type as $G$ is isomorphic to $\PGO^+(q')$ for some quadratic form $q'$ of the same dimension as $q$. Since the Dynkin diagram of $G$ has no automorphism, $f$ must be the identity map, and condition (1) in the definition above means that the quadrics $X_q$ and $X_{q'}$ have isomorphic motives modulo $2$. By a result of Izhboldin~\cite{I}, under this condition, the quadratic forms $q$ and $q'$ are similar. Hence the groups $G$ and $G'$ are isomorphic and condition (2) holds. 
\end{ex}

For other types of groups, motivic equivalence generally differs from isomorphism. This was already noticed by Izhboldin~\cite{I} for even-dimensional quadratic forms; see also \S~\ref{iso.section} below. Nevertheless, we aim at proving that all absolutely almost simple classical groups admit critical varieties, at least for $p=2$, which is the main prime of interest for such groups. In other words,  motivic equivalence modulo $2$ for classical groups can be checked from the isomorphism of standard motives for a given $\Theta\subset\Delta$.

We therefore assume untill the end of this section that $p=2$ and we sometimes write 'motivic equivalence' for motivic equivalence mod $2$. The main result of this section is the following : 

\begin{thm}\label{main.thm}
All absolutely almost simple groups $G$ of classical type admit a critical variety modulo $2$. More precisely considering motivic equivalence modulo $2$, the following twisted flag $G$-varieties are critical :
\begin{enumerate}[label=\roman*)]
    \item If $G$ is of type $\sfA_n$ (with $n\geq 2$), $X_{\{1,n\},G}$ is critical.
    \item If $G$ is of type $\sfB_n$, $X_{\{1\},G}$ is critical.
  \item If $G$ is of type $\sfC_n$ (with $n\geq 2$), $X_{\{2\},G}$ is critical.
    \item If $G$ is of type $\sfD_n$ (with $n\geq 3$), $X_{\{1\},G}$ is critical.
\end{enumerate}
\end{thm}

\begin{rem}
If $n=1$, the group $G$ has a unique twisted flag variety, which obviously is critical; therefore, we may assume $n\geq 2$ in the sequel. 
\end{rem}

All groups of classical type can be described in terms of some  algebra with involution, which is unique except in some low degree cases, see~\cite[\S\,15 \& 42]{KMRT}. Using the description of projective homogeneous varieties in~\cite[\S\,2.4]{MT}, we give the following uniform notation for the varieties appearing in the above theorem:

\begin{defi}
\label{descriptioncritical.def}
Let $(A,\sigma)$ be an algebra with involution over $F$, such that $G=\psim(A,\sigma)$ is as in Theorem~\ref{main.thm}. We let $X_\sigma$ be the $F$-variety defined as follows, depending on the type and degree of $(A,\sigma)$:
\begin{itemize}[label=-]
\item \textit{$\sfA_n$ (with $n\geq 2$).} If $\sigma$ is unitary, and $\deg A \geq 3$, $X_{\sigma}=X_{\{1, \deg A -1\},G}$ is the variety of flags of right ideals $I_1\subset I_2$ of respective reduced dimension $1$ and $\deg(A)-1$ and such that $\sigma(I_1)I_2=\{0\}$. 
\item \textit{$\sfB_n$ (with $n\geq 1$) and $\sfD_n$ (with $n\geq 3$).} If $\sigma$ is orthogonal with $\deg A \geq 3$ and $\deg A\not =4$, $X_{\sigma}=X_{\{1\},G}$ is the variety of isotropic right ideals in $A$ of reduced dimension $1$.
\item \textit{$\sfC_n$ (with $n\geq 2$).} If $\sigma$ is symplectic and  $\deg A \geq 4$,  $X_{\sigma}=X_{\{2\},G}$ is the variety of isotropic right ideals in $A$ of reduced dimension $2$.
\end{itemize}
\end{defi}

\begin{rem}
In the orthogonal case, the variety $X_\sigma$ is the so-called involution variety, previously considered by Tao~\cite{tao}. If in addition the underlying algebra is split, so that $(A,\sigma)=\Ad_q$ for some quadratic form $q$, then $X_\sigma$ is isomorphic to the projective quadric associated to $q$.
\end{rem}

\begin{rem}
The case of a degree $4$ algebra with orthogonal involution is excluded from our discussion, since the corresponding algebraic group, of type $\sfA_1+\sfA_1$, is not absolutely almost simple. In addition there are examples of groups of such type which admit no critical variety. Indeed, consider a quaternion division algebra $Q$ over $F$, and the algebras with involution described in~\cite[(15.2),(15.3)]{KMRT}\[(A,\sigma)=(Q,\ba)\otimes(Q,\ba)\mathrm{\ and\  } (B,\tau)=(Q,\ba)\otimes(M_2(F),\ba),\] where $\ba$ stands for the canonical involution, with corresponding groups
\[
\psim(A,\sigma)\simeq \PGL_1(Q)\times \PGL_2(F)\\ 
{\ \mathrm{ and}\ }\psim(B,\tau)\simeq \PGL_1(Q)\times\PGL_1(Q).\] 
The twisted flag varieties under $G=\psim(A,\sigma)$ and $G'=\psim(B,\tau)$ are the following :  $X_{\{1\},G}\simeq X_{\{2\},G}\simeq \SB(Q)$,  $X_{\Delta,G}\simeq \SB(Q)\times\SB(Q)$, $X_{\{1\},G'}\simeq \SB(Q)$, $X_{\{2\},G'}\simeq {\mathbb P}^1$ and $X_{\Delta,G'}\simeq \SB(Q)\times {\mathbb P}^1$. Therefore, $X_{\{1\},G}\simeq X_{\{1\},G'}\simeq X_{\{2\},G}$ and, by~\cite[Prop. 3.3]{dC-Upper}, $X_{\Delta,G}$ and $X_{\Delta,G'}$ have isomorphic motives modulo $2$. On the other hand, since $Q$ is division, the motive of $X_{\{2\},G}$ is indecomposable, while the motive of $X_{\{2\},G'}$ is a sum of Tate motives, so that $G$ and $G'$ are not motivic equivalent. This proves that none of the three varieties $X_{\{1\},G}, X_{\{2\},G}$ and $X_{\Delta,G}$ is critical for $G$. 
\end{rem}

The following definition extends motivic equivalence for quadratic forms to involutions: 
\begin{defi}
Let $(A,\sigma)$ and $(B,\tau)$ be two algebras with involution of the same type over $F$, and such that the corresponding groups $\psim(A,\sigma)$ and $\psim(B,\tau)$ are as in Theorem~\ref{main.thm}. The involutions $\sigma$ and $\tau$ are called motivic equivalent (denoted $\sigma \overset{{\tiny m}}{~\sim~} \tau$)  if the varieties $X_\sigma$ and $X_\tau$ have isomorphic motives with coefficients in $\mf_2$. 
\end{defi}

We claim that the main theorem follows from the following : 
\begin{thm}
\label{sousmain.thm}
Let $(A,\sigma)$ and $(B,\tau)$ be as in the definition above. If $\sigma$ and $\tau$ are motivic equivalent, then the corresponding groups $\psim(A,\sigma)$ and $\psim(B,\tau)$ are twisted forms of the same quasi-split group, and are motivic equivalent modulo $2$. 
\end{thm}

\begin{ex}
\label{spliteven.ex}
Let $q$ and $q'$ be two quadratic forms defined on the same even dimensional vector space $V$, and consider the split algebra $A=\End_F(V)$ endowed with the involutions $\sigma$ and $\tau$ respectively adjoint to $q$ and $q'$. The corresponding varieties are the quadrics $X_\sigma=X_q$ and $X_\tau=X_{q'}$. Therefore, $\sigma$ and $\tau$ are motivic equivalent precisely when the two quadratic forms $q$ and $q'$ are motivic equivalent. 
By a result of Vishik~\cite[Theorem 4.18]{vish-lens} (see also~\cite{K:motivic}), this holds if and only if $q$ and $q'$ have the same Witt index over any extension of the base field $F$. Under this condition, as explained in~\cite[Lemma 2.6]{K:motivic}, the two quadratic forms have the same discriminant, so that the groups $G=\PGO^+(q)$ and $G'=\PGO^+(q')$ are inner twisted forms of the same quasi-split group $G_0$. Moreover, from~\cite[Cor. 16]{dC}, they are motivic equivalent, that is, equality of the Witt indices of $q$ and $q'$ over any extension of the base field implies that non only the quadrics, but twisted flag varieties of all type under $G$ and $G'$ have isomorphic motives. In particular, Theorem~\ref{sousmain.thm} holds for the involutions adjoint to $q$ and $q'$. 

Note that this is not enough to prove that the quadric $X_q$ is critical for $\PGO^+(q)$, since there may exist non split algebras $B$ with orthogonal involution $\tau$ such that the corresponding group $\psim(B,\tau)$ also is an inner twisted form of the same quasi-split group $G_0$. So we still have to check what happens in this case (see Proposition~\ref{jpp} below). 
\end{ex}

\begin{proof}[Proof of theorem~\ref{main.thm}]
Assume Theorem~\ref{sousmain.thm} holds. Let $G$ be an adjoint group as in Theorem~\ref{main.thm}, and $(A,\sigma)$ an algebra with involution such that $G=\psim(A,\sigma)$. We need to prove that the variety $X_\sigma$ is critical for the group $\psim(A,\sigma)$. So, consider a group $G'$ which is an inner twisted form of the same quasi split group $G_0$ as $G$. In particular, there exists an algebra with involution $(B,\tau)$ of the same type as $(A,\sigma)$ such that $G'=\psim(B,\tau)$. 
If $G$ is not of type $^1\sfD_4$, any isomorphism $f$ of the Dynkin diagram of $G$ fixes the subset $\Theta$ defining the variety $X_\sigma$. Therefore, condition (1) of Definition~\ref{critical.def} holds if and only if $\sigma$ and $\tau$ are motivic equivalent. Under this condition, the groups $G$ and $G'$ are motivic equivalent by theorem~\ref{sousmain.thm} and this proves the variety $X_\sigma$ is critical. 

Assume now that $G$ and $G'$ have type $^1\sfD_4$, that is $(A,\sigma)$ and $(B,\tau)$ have degree $8$, orthogonal type and trivial discriminant. As explained in~\cite[\S 42.A]{KMRT}, there exists a triple of degree $8$ algebras with orthogonal involutions with trivial discriminant $\bigl((B,\tau),(C_+,\sigma_+),(C_-,\sigma_-)\bigr)$ such that 
$G'=\psim(B,\tau)=\psim(C_+,\tau_+)=\psim(C_-,\tau_-)$. The algebras with involution $(C_+,\tau_+)$ and $(C_-,\tau_-)$ are the two components of the Clifford algebra of $(B,\tau)$. Moreover, the automorphism group of the Dynkin diagram of $G'$ acts on this triple by permutation (loc. cit. (42.3)). Therefore, condition (1) in Definition~\ref{critical.def} now means that $\sigma$ is motivic equivalent to one of the three involutions $\tau$, $\tau_+$ and $\tau_-$. In all three cases, Theorem~\ref{sousmain.thm} implies that $G$ and $G'$ are motivic equivalent and this concludes the proof. 
\end{proof}

 In the next sections, we provide a proof of Theorem~\ref{sousmain.thm}, which can be thought of as a translation of our main result in terms of algebras with involution. Note that even though the critical varieties does depend on the group, the structure of our proofs for each type follow the same lines. More precisely we will use the generic index reduction fields introduced in~\S\ref{functionfield.section}, and we will prove we can control motivic equivalence on those fields through the study of motivic isomorphisms of some prescribed quadrics.

\subsection{First reductions}

From now on, we consider two algebras with involution $(A,\sigma)$ and $(B,\tau)$ of the same type over $F$. We assume in addition that the groups $G=\psim(A,\sigma)$ and $G'=\psim(B,\tau)$ are as in Theorem~\ref{main.thm}. As explained in  example~\ref{splitodd.ex}, Theorems~\ref{main.thm} and ~\ref{sousmain.thm} are already known in two cases. 
As we already mentioned in~\ref{splitodd.ex}, they do hold if $G$ has type $\sfB_n$, that is if $A$ and $B$ are split and $\sigma$ and $\tau$ are adjoint to some odd dimensional quadratic forms by~\cite{I}. 
If $G$ has type $^1\sf A_n$, that is if $A$ and $B$ have center $F\times F$ and $\sigma$ and $\tau$ are unitary, we have $(A,\sigma)\simeq (E_1\times E_1^{\mathrm{op}}, \varepsilon)$ and $(B,\tau)\simeq (E_2\times E_2^{\mathrm{op}},\varepsilon)$ for some degree $n+1$ central simple algebras $E_1$ and $E_2$ over $F$, and the corresponding automorphism groups are isomorphic to $\SL_1(E_1)$ and $\SL_1(E_2)$ respectively. Therefore the theorems also hold in this case, by~\cite[Thm. 19]{dC}. 
Therefore, we assume until the end of \S~\ref{proof.section} that the algebras $A$ and $B$ have degree at least $3$ if the involutions are unitary, and at least $4$ if they are symplectic. We assume in addition that they have even degree greater than or equal to $6$ if they are orthogonal. In the unitary case, we always assume in addition that the center is non split.

In orthogonal and unitary types, the groups $G$ and $G'$ might be of outer type, in which case they become of inner type over some quadratic extensions $K$ and $K'$ of the base field $F$. The fields $K$ and $K'$ are the respective centers of $A$ and $B$ in unitary type and the quadratic extensions corresponding to the discriminants $d_\sigma$ and $d_\tau$ in orthogonal type. In this section, we prove that if $\sigma$ and $\tau$ are motivic equivalent, then $K=K'$ and the $2$-primary parts of the algebras $A$ and $B$ generate the same subgroup in the Brauer group of their center. For orthogonal and symplectic involutions, this amounts to the following:

\begin{prop}\label{jpp}
Assume $(A, \sigma)$ and $(B, \tau)$ are of the same orthogonal or symplectic type over $F$. If $\sigma$ and $\tau$ are motivic equivalent, then $A$ and $B$ are isomorphic. In addition, in orthogonal type, $\sigma$ and $\tau$ have the same discriminant. \end{prop}

\begin{proof}
Assume $\sigma$ and $\tau$ are motivic equivalent, that is $X_\sigma$ and $X_\tau$ have isomorphic motives with coefficients in $\mf_2$. Extending scalars to an algebraic closure of the base field, we already get that the algebras $A$ and $B$ have the same degree. 
Moreover, since $A$ and $B$ are endowed with involutions of the first kind, their period is $2$. Therefore, to ensure that $A$ and $B$ are isomorphic, it is enough to prove they generate the same subgroup of the Brauer group of $F$.\\

\noindent {\it Symplectic case}. Let $F_A$ be the function field of the Severi-Brauer variety of $A$. Since $\sigma$ is symplectic, the group $G=\psim(A,\sigma)$ is split over $F_A$, and the motive $M(X_{\sigma}\times_F F_A)$ is a direct sum of Tate motives. Therefore, $M(X_{\tau}\times_F F_A)$ also is a direct sum of Tate motives, hence it is isotropic. 
So, we may apply \cite[Corollary 15.9]{K2}, which provides a motivic decomposition of $X_{\tau}\times_F{F_A}$ containing an indecomposable direct summand isomorphic to a twist of $M(\SB(D))$, where $D$ is the division $F_A$-algebra Brauer equivalent to $B_{F_A}$.
Since $M(\SB(D))$ is a Tate motive if and only if $D$ is split, we get that $B$ is split over $F_A$. 

By the same argument $A$ is also split over the function field of the Severi-Brauer variety of $B$. Hence $A$ and $B$ generate the same subgroup in $\Br(F)$ by Amitsur's theorem \cite[Theorem 9.3]{amitsur}, and the result is proved in this case.  \\

\noindent {\it Orthogonal case}. If $\sigma$ and $\tau$ are motivic equivalent, then in particular $X_\sigma$ and $X_\tau$ have the same Chow groups with $\mf_2$ coefficients. Therefore, the cokernel of the following maps coincide, where $\bar F$ is an algebraic closure of $F$:
$$\coker\bigl(\CH^1(X_\sigma)\rightarrow \CH^1(X_\sigma\times_F \bar F)\bigr)=\coker\bigl(\CH^1(X_\tau)\rightarrow \CH^1(X_\tau\times_F \bar F)\bigr).$$
On the other hand, as explained in~\cite[Proof of Thm 4.8]{tao}, those cokernels are respectively isomorphic to the subgroups of the Brauer group of $F$ generated by $[A]$ and $[B]$. Hence the algebras $A$ and $B$ generate isomorphic subgroups of the Brauer group, over the base field $F$ and also over any extension of $F$. In particular, each of them is split by the function field of the Severi Brauer variety of the other, and, since they have exponent $2$, it follows $A$ and $B$ are isomorphic. 

It remains to prove that the respective discriminants $d_\sigma$ and $d_\tau$ are equal, or equivalently that the corresponding quadratic extensions $K$ and $K'$ are isomorphic. 
Consider the variety $\mathfrak{X}=X_{\Delta,G}\times X_{\Delta,G'}$ which is the direct product of the varieties of Borel subgroups of both groups $G$ and $G'$. Since $F$ is quadratically closed in the function field $\cf$ of $\mathfrak{X}$, $K$ and $K'$ induce quadratic extensions of $\cf$, which we still denote by $K$ and $K'$, and it is enough to prove they are isomorphic over $\cf$. 
By definition of $\mathfrak{X}$, the groups $G$ and $G'$ are quasisplit over $\cf$, that is $A_\cf$ is split and the involutions $\sigma_\cf$ and $\tau_\cf$ are respectively adjoint to the quadratic forms $q=r\mathbb{H}\oplus\qform{1,-d_\sigma}$ and $q'=r\mathbb{H}\oplus\qform{1,-d_\tau}$, where $\mathbb{H}$ is a hyperbolic plane. Therefore, the motives of $X_\sigma$ and $X_\tau$ respectively decompose over $\cf$ as a sum or Tate motives plus a summand isomorphic to the motive of $\Spec K$ for $X_\sigma$ and $\Spec K'$ for $X_\tau$. Since $\sigma$ and $\tau$ are motivic equivalent, it follows that $K$ and $K'$ are isomorphic and this finishes the proof. 
\end{proof}

The analogue of Proposition \ref{jpp} for unitary involutions is the following: 

\begin{prop}\label{jpp2}
Assume $(A, \sigma)$ and $(B, \tau)$ are of unitary type. 
Let $K$ and $K'$ be the respective centers of $A$ and $B$. If $\sigma$ and $\tau$ are motivic equivalent, then $K$ and $K'$ are isomorphic over $F$. Moreover, under the identification $\Br(K) \simeq \Br(K')$, the $2$-components of $A$ and $B$ generate the same subgroup in $\Br(K)$. 
\end{prop}

\begin{proof}
Assume $\sigma$ and $\tau$ are motivic equivalent. Again, this implies the algebras $A$ and $B$ have the same degree. 
In order to prove that $K$ and $K'$ are isomorphic, we can use the same strategy as in the orthogonal case. 
Consider the variety $\mathfrak{X}=X_{\Delta,G}\times X_{\Delta,G'}$ which is the direct product of the varieties of Borel subgroups of both groups $G$ and $G'$. Since $F$ is quadratically closed in the function field $\cf$ of $\mathfrak{X}$, $K$ and $K'$ induce quadratic extensions of $\cf$, which we still denote by $K$ and $K'$, and it is enough to prove they are isomorphic over $\cf$. By definition of $\mathfrak{X}$, the groups $G$ and $G'$ are quasisplit over $\cf$. Therefore, we may apply~\cite[Cor 7.2]{K-ugr}, which says that the motives of $X_{\sigma}$ and $X_{\tau}$ over $\cf$ respectively decompose into a sum of Tate motives plus summands isomorphic to shifts of the $F$-motive of $\Spec K $ for $X_\sigma$ and $\Spec K' $ for $X_\tau$. Therefore, again the fields $K $ and  $K'$ are isomorphic over $\cf$ hence over $F$. 
From now on we identify $K$ and $K'$. 

After scalar extension to $K$ we have
$$(X_{\sigma})_K \simeq X(1 , \deg(A)-1; A), \, \, \, (X_{\tau})_K \simeq X(1 , \deg(B)-1; B) \, , $$
\noindent where $X(1 , \deg(A)-1; A)$ (resp. $X(1 , \deg(B)-1; B)$) is the variety of flags $I_1 \subset I_2$ of right ideals $I_1$ and $I_2$ in $A$ (resp. in $B$) of reduced dimension $1$ and reduced codimension $1$ (see \cite[Proposition 2.15]{KMRT} and \cite[Lemma 15.5]{K2}).
Since $\sigma$ and $\tau$ are motivic equivalent, those varieties have isomorphic motives with $\mf_2$ coefficients; so in particular, they also have isomorphic upper motives. Since $X(1 , \deg(A)-1; A)$ has a rational point over a field extension $L$ of $K$ if and only if the Severi-Brauer variety $\SB(A)$ also does, their upper motives are isomorphic by~\cite[Cor 2.15]{K-inner}. Similarly, the upper motive of $X(1 , \deg(B)-1; B)$ is isomorphic to that of $\SB(B)$ and we get that the Severi-Brauer varieties $\SB(A)$  and $\SB(B)$ have isomorphic upper motives. Since we work with $\mf_2$ coefficients, it follows by~\cite[Thm 1]{dC-An} that the $2$-primary parts of the algebras $A$ and $B$ generate the same subgroup of the Brauer group of $K$. 
\end{proof}

\begin{rem} 
\label{algebra.rem}
The condition on the algebras $A$ and $B$ in~Proposition~\ref{jpp2} above is optimum. In particular, the algebras $A$ and $B$ generally are non-isomorphic. This was already observed in~\cite{dC} for groups of inner type. More precisely, 
let $E$ be a division algebra of exponent $8$ or a larger $2$-power, and consider the division algebra $E'$ Brauer equivalent to $E^3$. The algebras $E$ and $E'$ are $2$-primary, and generate the same subgroup of the Brauer group of $F$. Therefore, as explained in~\cite[\S X.1]{dC} the groups $\SL_1(E)$ and $\SL_1(E')$ are motivic equivalent. Hence the exchange involutions respectively defined on $A=E\times E^\mathrm{op}$ and $B=E'\times {E'}^\mathrm{op}$ are motivic equivalent. Nevertheless, $E'$ is isomorphic neither to $E$ nor to its opposite, so $A$ and $B$ are non-isomorphic.  
\end{rem}

\subsection{Motivic equivalence and generic index reduction fields}
The main result of this section is the following proposition, which reduces the proof of Theorem~\ref{sousmain.thm} to some low Schur index cases, as we will explain in~\S\ref{proof.section}. This result is a consequence of~\cite[Thm. 16]{dC}, which characterizes motivic equivalent algebraic groups in terms of their higher $2$-Tits indexes. 

\begin{prop}
\label{equiv.prop}
Assume $(A,\sigma)$ and $(B,\tau)$ are of the same type $t$ and have the same degree. If $t=o$, we assume $\sigma$ and $\tau$ have the same discriminant; 
if $t=u$ we assume $A$ and $B$ have isomorphic centers.  
In all three types, we assume in addition that the $2$-primary parts of $A$ and $B$ generate the same subgroup of the Brauer group of their center. 

Consider the associated quadratic forms $\cq_\sigma$ and $\cq_\tau$ over the function fields $\cf_{A,t}$ and $\cf_{B,t}$, as defined in \S\,\ref{functionfield.section}. 
Denoting by $\cf$ the free composite of $\cf_{A,t}$ and $\cf_{B,t}$, the following assertions are equivalent.
\begin{enumerate}
\item $\psim(A,\sigma)$ and $\psim(B,\tau)$ are motivic equivalent; 
\item $\psim(A,\sigma)_\cf$ and $\psim(B,\tau)_\cf$ are motivic equivalent; 
\item $(\cq_\sigma)_{\cf}$\!$\overset{{\tiny m}}{~\sim~}$$(\cq_\tau)_{\cf}$.
\end{enumerate}
\end{prop}

The proof uses the following lemma: 
\begin{lem} (cf~\cite[Thm. 16]{dC}) 
\label{Charles.lem}
Let $(A,\sigma)$ and $(B,\tau)$ be as in Proposition~\ref{equiv.prop}. In particular, we have $A\simeq B$ in orthogonal and symplectic type. The groups $G=\psim(A,\sigma)$ and $G'=\psim(B,\tau)$ are motivic equivalent if and only if for all field extensions $M/F$ we have
$$i_{w,2}(\sigma_M)=i_{w,2}(\tau_M).$$
\end{lem} 

\begin{proof}
The assumptions we made guarantee that the two groups are inner twisted forms of the same quasi split group. Therefore, we may apply~\cite[Thm. 16]{dC}, which says that they are motivic equivalent if and only if they have the same higher Tits $2$-indexes, that is if and only if for all extension $M$ of $F$, the groups $G_M$ and $G'_M$ have the same Tits $2$-index. As explained in~\cite[\S II]{dCG}, the Tits $2$-index of $G$ is empty if the group is anisotropic, and determined by the Schur index of the $2$ primary part of $A$ and the Witt $2$-index $i_{w,2}(\sigma)$ otherwise, except possibly in inner type $\sfD_n$, where we need to specify which of the two extreme vertices belongs to the Tits index when only one of them does.  
Since the $2$-primary parts of $A$ and $B$ generate the same subgroup of the Brauer group of their center, they have the same Schur index over any field extension of their center. Therefore the lemma is proved, except possibly in inner type $\sfD_n$. 

Let us now study this case in more details. From our hypothesis, we already know that the Tits $2$-indexes of both groups intersect $\{n-1,n\}$ in two subsets of the same cardinality. Therefore, at least one of the two varieties $X_{\{n-1\},G'}$ and $X_{\{n\},G'}$ is dominated modulo $2$ by $X_{\{i\},G}$ for $i\in\{n-1,n\}$, in the sense of~\cite[d\'ef. 4]{dC}. Conversely, at least one of $X_{\{n-1\},G}$ and $X_{\{n\}, G}$ is dominated (modulo $2$) by $X_{\{j\},G'}$, for $j\in\{n,n-1\}$.
We claim that that there is a choice of $i\in\{n-1,n\}$ and $j\in\{n-1,n\}$ such that the varieties $X_{\{i\},G}$ and $X_{\{j\},G'}$ are equivalent modulo $2$. Indeed, assume for the sake of contradiction that this is not the case. Then, up to renaming the vertices, $X_{\{n-1\},G'}$ dominates $X_{\{n-1\},G}$, $X_{\{n\},G}$ dominates $X_{\{n-1\},G'}$, $X_{\{n\},G'}$ dominates $X_{\{n\},G}$ and $X_{\{n-1\},G}$ dominates $X_{\{n\},G'}$. By transitivity, we then get that $X_{\{n\},G'}$ also dominates $X_{\{n-1\},G}$, a contradiction since $X_{\{n\},G'}$ and $X_{\{n-1\},G}$ as well as $X_{\{n-1\},G'}$ and $X_{\{n\},G}$ would then be equivalent modulo $2$. Combining this with the assumptions of the Lemma, we get that the groups $G$ and $G'$ are equivalent modulo $2$ in the sense of~\cite[D\'efinition 6]{dC}, hence motivic equivalent modulo $2$ by~{\em loc.\,cit. }Th\'eor\`eme 13.
\end{proof}

\begin{proof}[Proof of proposition~\ref{equiv.prop}]
By scalar extension, (1) clearly implies (2). The converse follows from Lemma~\ref{Charles.lem}. Indeed, let us assume (2) holds. Since motivic equivalence modulo $2$ can be checked after odd-degree field extensions, we may assume in the unitary case that both $A$ and $B$ are of $2$ primary index. Consider a field extension $M/F$, and set $\mathcal{M}$ for the compositum of $M$ and $\cf$, that is the free composite of the function fields of the relevant varieties for $\sigma_M$ and $\tau_M$.
Since $\cm$ is an extension of $\cf$, by Lemma~\ref{Charles.lem}, (2) implies  $i_{w,2}(\sigma_\mathcal{M})=i_{w,2}(\tau_\mathcal{M})$. 
In addition, the condition we made on the algebra $A$ and $B$ guarantee that the field $\cf$ is a purely transcendental extension of $\cf_{A,t}$ and $\cf_{B,t}$, and similarly over $M$. Therefore, applying Lemma~\ref{Wittindex.lemma} to $\sigma_M$ and $\tau_M$, we get
\[i_{w,2}(\sigma_M)=i_{w,2}(\sigma_{\cm})=i_{w,2}(\tau_\cm)=i_{w,2}(\tau_M).\]
It follows that the Witt $2$-indexes of $\sigma_M$ and $\tau_M$ coincide for any field extension $M/F$, which proves (1). 

{We now show that (2) and (3) are equivalent. If the involutions are orthogonal, this follows from Example~\ref{spliteven.ex} since $A_{\cf}$ and $B_\cf$ are split and $\sigma_\cf$ and $\tau_\cf$ are respectively adjoint to $\cq_\sigma$ and $\cq_\tau$. 
Assume now that $\sigma$ and $\tau$ are either symplectic or unitary. 
By Lemma~\ref{Charles.lem}, (2) holds if and only if $\sigma$ and $\tau$ have the same $2$-Witt indices over any extension $\cm$ of $\cf$. In view of Proposition \ref{lowindex.prop} this is equivalent to equality of the Witt indices of the quadratic forms $\cq_{\sigma_{\mathcal{M}}}$ and $\cq_{\tau_{\mathcal{M}}}$ for all $\cm$. As explained in Example~\ref{spliteven.ex}, this in turn characterizes motivic equivalence of $\cq_\sigma$ and $\cq_\tau$. 
}
\end{proof}

\subsection{Proof of Theorem~\ref{sousmain.thm}}
\label{proof.section}
With this in hand, we can now prove Theorem~\ref{sousmain.thm} which asserts that if the involutions $\sigma$ and $\tau$ are motivic equivalent, then the groups $\psim(A,\sigma)$ and $\psim(B,\tau)$ are motivic equivalent. 

So, assume $\sigma$ and $\tau$ are motivic equivalent, that is  $X_\sigma$ and $X_\tau$ have isomorphic motives modulo $2$. By Propositon~\ref{jpp}, the algebras $A$ and $B$ are isomorphic if the involutions are orthogonal or symplectic. Using in addition Proposition~\ref{jpp2}, observe that in all three types, the algebras satisfy the conditions of Proposition \ref{equiv.prop}. Hence, we may extend scalars to $\cf$, or equivalently, we may assume that $A=B$, and it is a split algebra in orthogonal and unitary type, and has index at most $2$ in symplectic type. 

If the involutions are orthogonal, the result now follows from Example~\ref{spliteven.ex}. 
In unitary and symplectic type, we use Lemma~\ref{Charles.lem}, so we need to prove that $i_{w,2}(\sigma_M)=i_{w,2}(\tau_M)$ for all extension $M$ of the base field $F$. Note that because of the condition on the algebra, the Witt $2$-indices and the Witt indices coincide for $\sigma$ and $\tau$ (see~Proposition~\ref{lowindex.prop}). 
In unitary type, as explained in~\cite[(6.3)]{KMRT}, if the quadratic field $K$ splits over $F$, then $\sigma$ and $\tau$ are hyperbolic, hence have maximal Witt index. The same holds in symplectic type if $A_M$ is split. Since $\sigma$ and $\tau$ are motivic equivalent, the varieties $X_\sigma$ and $X_\tau$ have isomorphic motives over the base field $F$, and also over any extension $M$ of $F$. Therefore, the following proposition gives the required equality in all other cases. 

\begin{prop}
\label{motoviclowindex.prop}
{Let $(A,\sigma)$ be an algebra with unitary or symplectic involution. We assume $A$ has index $1$ in unitary type, and $A$ has index $2$ in symplectic type. 
If $\sigma$ is unitary, we assume in addition that $K/F$ is a field extension.} 
Then the complete motivic decomposition of $X_\sigma$ modulo $2$ determines $i_{w}(\sigma)$.
\end{prop}
\begin{proof}
{
Assume first that $\sigma$ is unitary and $A$ is split. By \cite[Lemma 3.1]{K-uni}, for every $i < d/2$ the following holds : $i_w(\sigma)> i$ if and only if the complete motivic decomposition of $X_{\sigma}$ contains the Tate motive $\mathbb{F}_2(2i)$. So the result is proved in this case. }

Assume now $\sigma$ is 
a symplectic involution of the algebra $M_m(H)$, for some quaternion division algebra $H$. The involution $\sigma$ is adjoint to a rank $m$ hermitian form, denoted by $h_\sigma$, defined on the $H$-module $V\simeq H^m$, and with values in $(H,\ba)$, where $\ba$ stands for the canonical involution of $H$. The variety $X_{\sigma}$, is isomorphic to the variety $X(2, (V,h_\sigma))$ of totally isotropic $H$-submodules of $V$ of reduced dimension $2$.
If the hermitian form $h_\sigma$ is isotropic, it decomposes as 
$$(V,h_\sigma)= \mathbb{H}(H) \perp (W, h') \, ,$$ where $\mathbb{H}(H)$ is a hyperbolic plane over $(H,\ba)$. 
 Applying \cite[Corollary 15.9]{K2} to the above decomposition, we obtain a decomposition of the motive of $X_{\sigma}$ as a sum of shifts of the following: two Tate motives, two copies of the motive of the product $\SB(H)\times X(1, (W,h'))$, and the motives of $\SB(H)$ and $X(2, (W,h'))$, where $\SB(H)$ denotes the Severi-Brauer variety of $H$. Note that, since $\ind H=2$, there is no Tate motives in the complete motivic decompositions of $\SB(H)\times X(1, (W,h'))$ and $\SB(H)$. Therefore, denoting by $N(X)$ the number of Tate motives in the complete motivic decomposition of $X$, we get that 
$$N(X_\sigma)= N\bigl(X(2, (V,h_\sigma))\bigr) = 2 + N\bigl(X(2, (W,h'))\bigr)\,.$$
Since in addition $i_w(h_\sigma)=N(X_\sigma)=0$ if $h_\sigma$ is anisotropic, an  induction argument shows that $i_w(\sigma)=2 i_w(h_\sigma) = 2 N(X_{\sigma})$.              {This concludes the proof. }
\end{proof}

\subsection{Generalization of Vishik's theorem}

Using the material developped in this section, we may extend Vishik's celebrated criterion of motivic equivalence~\cite[Theorem 4.18]{vish-lens} (see also~\cite{K:motivic}) from quadratic forms to involutions: 
\begin{cor}
\label{Vishik.cor}
If $(A, \sigma)$ and $(B, \tau)$ are of the same orthogonal or symplectic type over $F$, the involutions $\sigma$ and $\tau$ are motivic equivalent if and only if we have
$$ A \simeq B \, \, \text{  and   } \, \, \, i_{w,2}(\sigma_M)=i_{w,2}(\tau_M)
\mbox{ for all field extensions }M/F.$$

If $(A,\sigma)$ and $(B,\tau)$ are of unitary type, then $\sigma$ and $\tau$ are motivic equivalent if and only if the centers of $A$ and $B$ are isomorphic over $F$, the $2$-primary parts of $A$ and $B$ generate the same subgroup of the Brauer group of their center, and
$$ i_{w,2}(\sigma_M)=i_{w,2}(\tau_M)\mbox{ for all field extensions }M/F. $$
\end{cor}
\begin{proof}
Assume first that the involutions $\sigma$ and $\tau$ are motivic equivalent. By Theorem~\ref{sousmain.thm}, the groups $\psim(A,\sigma)$ and $\psim(B,\tau)$ are motivic equivalent. The result follows combining 
Propositions \ref{jpp} and~\ref{jpp2}, which prove the conditions on $A$ and $B$ hold, and Lemma~\ref{Charles.lem} which provides equality of the higher Witt $2$-indices of $\sigma$ and $\tau$. 

Let us now prove the converse. Assume first the involutions $\sigma$ and $\tau$ are symplectic or unitary. All the assumptions of Lemma~\ref{Charles.lem} are satisfied, and we get that the groups $\psim(A,\sigma)$ and $\psim(B,\tau)$ are motivic equivalent. It follows that $\sigma$ and $\tau$ are motivic equivalent, since the subset defining the varieties $X_\sigma$ and $X_\tau$ is fixed under any Galois invariant automorphism of the Dynkin diagram of the underlying groups, so the result is proved in these cases.  In orthogonal type, since $i_{w,2}(\sigma_M)=i_{w,2}(\tau_M)$ for all field extensions $M/F$, we claim the involutions $\sigma$ and $\tau$ have the same discriminant. Indeed, as explained in the proof of Proposition~\ref{functionfield.prop}, we may compute the discriminants $d_\sigma$ and $d_\tau$ after extending scalars to the generic splitting field $\cf_{A,o}$ of the algebra $A$. Over this field, $\sigma$ and $\tau$ are adjoint to some quadratic forms $\cq_\sigma$ and $\cq_\tau$ that have the same Witt indices over all extensions of $\cf_{A,o}$. Therefore, they have the same discriminant by~\cite[Lemma 2.6]{K:motivic}, and this proves $d_\sigma=d_\tau$. Hence again all the assumptions of Lemma~\ref{Charles.lem} are satisfied, and from the proof of this Lemma, we get that the groups $\psim(A,\sigma)$ and $\psim(B,\tau)$ are motivic equivalent for a choice of an identification between the Dynkin diagrams satisfying $f(1)=1$. So the involutions $\sigma$ and $\tau$ are motivic equivalent as required. 
Alternately, one may finish the proof using triality in type $^1\sfD_4$, which is the only type for which there exists $f$ which do not satisfy $f(1)=1$. Indeed, if the groups are motivic equivalent, the involution $\sigma$ is motivic equivalent either to $\tau$, or the the canonical involution of one of the two components of the Clifford algebra $(C_+,\tau_+)\times(C_-,\tau_-)$ of $(A,\tau)$. Assume for instance $\sigma$ is motivic equivalent to $\tau_+$. Then by Proposition~\ref{jpp}, $A\simeq C_+$. Therefore, by~\cite[(42.7)(1)]{KMRT}, the algebra $C_-$ is split, so the involution $\tau_-$ is adjoint to an $8$ dimensional quadratic form $q$ with trivial discriminant. Hence, by triality~\cite[(42.3)]{KMRT}, $(A,\tau)$ and $(A,\tau_+)$ are the two components of the even Clifford algebra of $q$. So they are isomorphic as algebras with involution, and we get that $\sigma$ also is motivic equivalent to $\tau$ as required.

\end{proof}

\subsection{Examples of non-critical varieties}

Projective homogeneous varieties under an absolutely almost simple algebraic group are not always critical. An easy explicit example is obtained by considering the variety of rank $1$ isotropic ideals for an algebra with symplectic involution. Indeed, as explained in~\cite[\S II.4 and \S III]{dCG}, there exists a field $F$, an algebra $A$ over $F$,  and two symplectic involutions $\sigma$ and $\tau$ such that the corresponding algebraic groups have different $2$-Tits indexes. In particular, $\sigma$ and $\tau$ are not motivic equivalent. Over a splitting field of $A$, both involutions are adjoint to a skew-symmetric bilinear form, hence all lines are isotropic. It follows that the varieties $X_{\{1\},\sigma}$ and $X_{\{1\},\tau}$ are both isomorphic to the Severi-Brauer variety of $A$, hence they have isomorphic Chow motives, and this proves $X_{\{1\},\sigma}$ is not critical. 

Excellent quadratic forms provide a more interesting example. Pick a field $F$ and $a,b,c,d\in F^\times$ such that the Pfister forms $\pform{a,b,c}$ and $\pform{a,b,d}$ are anisotropic and non-isomorphic. Consider the Pfister neighbors \[\varphi=\qform{1}+\qform{-c}\pform{a,b}\mbox{ and }\varphi'=\qform{1}+\qform{-d}\pform{a,b}.\]
Both are excellent quadratic forms by~\cite[Thm 28.3]{EKM}, since they are Pfister neighbors of the $3$-fold forms $\pform{a,b,c}$ and $\pform{a,b,d}$ respectively, with the same complement form $\qform{-a,-b,ab}$, which in turn is a Pfister neighbor of $\pform{a,b}$ with one-dimensional complement. 

Moreover, they are odd-dimensional and non similar, so they are not motivic equivalent by~\cite{I} (see Example~\ref{splitodd.ex}). 
Nevertheless, we claim that the varieties of complete flags of isotropic subspaces for $\varphi$ and $\varphi'$ have isomorphic motives. 
Indeed, by~\cite[Cor 2.15]{K-inner}, since for all extension $K/F$ we have 
\[\varphi_K\mbox{ is split}\ \Leftrightarrow\  \varphi'_K \mbox { is split}\ \Leftrightarrow \pform{a,b}_K\mbox{ is hyperbolic},\] 
their upper motives are isomorphic. 
In addition, those two varieties are generically split, hence they have isomorphic motives by~\cite[Thm 5.17]{PSZ}.

More generally, let $\varphi$ and $\varphi'$ be two anisotropic excellent quadratic forms of the same odd dimension, with associated decreasing sequences of Pfister forms $\rho_1 \supset \rho_2 \supset ... \supset \rho_{r-1}$, and $\rho'_1\supset\rho'_2\supset...\supset\rho'_{r-1}$ respectively, as in~\cite[Thm. 28.3]{EKM}. Note that, since $\varphi$ and $\varphi'$ are odd-dimensional, the last term in the sequence of Pfister complement forms have dimension $1$ (that is, with the same notations as in \textit{loc. cit.}, $\dim(\varphi_r)=1=\dim(\varphi'_r)$). 
It follows from \cite[Corollary 7.2]{Kar-Mer} that $\varphi$ and $\varphi'$ are motivic equivalent if and only if $\rho_i\simeq \rho_i'$ for all $i \in [1, r-1]$, while one may prove that the varieties of complete flags have isomorphic motives as soon as $\rho_{r-1}\simeq \rho'_{r-1}$.

\section{Motivic equivalence and isomorphism}
\label{iso.section}

In general, motivic equivalent involutions are not isomorphic. 
As noticed by Izhboldin in~\cite{I}, this may happen already in the split orthogonal case; indeed, there exists even dimensional motivic equivalent quadratic forms that are non-similar. Besides, in the unitary case, two motivic equivalent involutions might be defined on some non-isomorphic algebras. 

Nevertheless, under some condition on the base field, motivic equivalence for two involutions defined on the same algebra does imply isomorphism as we proceed to show.  
To be more precise, recall first that using Bayer and Parimala's proof of the Hasse principle conjecture II~\cite{BPAnnals}, Lewis and Tignol gave necessary and sufficient conditions on the base field $F$ under which cohomological invariants and signatures are enough to classify involutions on a central simple algebra (see~\cite{LT}). In this section, we prove that over a field $F$ satisfying those conditions, motivic equivalent involutions are isomorphic. This extends a previous result on quadratic forms du to Hoffmann~\cite{Hoffmann}.  The key observation is that motivic equivalent involutions have the same invariants. This is proved by Hoffmann~\cite{Hoffmann} in the split orthogonal case, and extends to other cases by our Proposition~\ref{equiv.prop}. More precisely, we have : 
\begin{prop}
\label{invariants.prop}
Let $(A,\sigma,\tau)$ be an algebra with two involutions of the same type over $F$.
If $\sigma$ and $\tau$ are motivic equivalent, then the following hold :
\begin{enumerate}
\item If $A$ has even degree and the involutions are orthogonal, \[d_\sigma=d_\tau\in \sq F\mbox{ and }c(\sigma_K)=c(\tau_K)\in\br(F)/\qform{[A]},\] where $K/F$ is the discriminant quadratic extension;
\item If $A$ has even degree and the involutions are unitary, \[[\cd(\sigma)]=[\cd(\tau)]\in\br(F);\]
\item In all three types, $\sign_P(\sigma)=\sign_P(\tau)$ for all $P\in X_F$.
\end{enumerate}
\end{prop}
\begin{proof}
If $\sigma$ and $\tau$ are motivic equivalent, then by Proposition~\ref{equiv.prop}, the quadratic forms $\cq_\sigma$ and $\cq_\tau$ also are motivic equivalent.
Hence, by~\cite[Cor.2.2, Lem.3.1]{Hoffmann}, we have $d(\cq_\sigma)=d(\cq_\tau)$, $c\bigl((\cq_\sigma)_{\ck}\bigr)=c\bigl((\cq_\tau)_{\ck}\bigr)$, where $\ck/\cf_{A,t}$ is the discriminant quadratic extension, and $|\sign_Q(\cq_\sigma)|=|\sign_Q(\cq_\tau)|$ for all $Q\in X_{\cf_{A,t}}$. The result follows by Proposition~\ref{functionfield.prop}.
\end{proof}

With this in hand, we get that motivic equivalent involutions actually are isomorphic over any field $F$ over which involutions are classified by their invariants. Let $F$ be a formally real field. Recall that the space of orderings $X_F$ is a topological space. Moreover, for all $a\in F^\times$, the so-called Harrison set \[H(a) = \{P \in X_F, a >_P 0\}\] is both closed and open in $X_F$ (see eg~\cite{P}). 
The field  $F$ is called SAP if, conversely, each closed and open subset of $X_F$ is a Harrison set, that is any prescription of signs at each ordering given by a partition of $X_F$ into two closed and open subsets is attained by some $a\in F^\times$. If in addition all formally real quadratic extensions of $F$ are SAP, the field $F$ is called ED. 
Applying the classification theorems given by Lewis and Tignol in~\cite{LT}, we get :
\begin{thm}
Let $F$ be either a non formally real field of cohomological dimension $\leq 2$, or a formally real field with virtual cohomological dimension $\leq 2$ and satisfying the ED property.
Let $(A,\sigma,\tau)$ be an algebra with two involutions of the same type over $F$.
If $\sigma$ and $\tau$ are motivic equivalent, then they are isomorphic.
\end{thm}

\begin{proof} The theorem follows immediately from~\cite[Thm.A \& Thm.B]{LT}, up to the following lemma.
\end{proof}
\begin{lem}
Let $\sigma$ and $\tau$ be two orthogonal involutions on a central simple algebra $A$ over $F$. Their Clifford algebras $\cc(\sigma)$ and $\cc(\tau)$ are isomorphic as $F$-algebras if and only of $d_\sigma=d_\tau$ and $c(\sigma_K)=c(\tau_K)$, where $K/F$ is the discriminant quadratic extension.
\end{lem}
\begin{proof}
Assume first that $\cc(\sigma)$ and $\cc(\tau)$ are isomorphic as $F$-algebras. They have isomorphic centers, therefore $d_\sigma=d_\tau$. Moreover, extending scalars from $F$ to $K$, we have $\cc(\sigma_K)\simeq\cc(\tau_K)$, where $K/F$ is the quadratic discriminant extension. Hence
$\cc_+(\sigma_K)\times\cc_-(\sigma_K)$ and  $\cc_+(\tau_K)\times\cc_-(\tau_K)$ are isomorphic as $K$-algebras. It follows that
$\cc_+(\sigma_K)$ is isomorphic either to $\cc_+(\tau_K)$ or to $\cc_-(\tau_K)$, and in both cases, we get $c(\sigma_K)=c(\tau_K)\in\br(K)/\qform{[A_K]}$.

Assume conversely that $d_\sigma=d_\tau$ and $c(\sigma_K)=c(\tau_K)\in\br(K)/\qform{[A_K]}.$ Since the center of $\cc(\sigma)$ is $K$, the Clifford algebra of $\sigma_K$ is
\[\cc(\sigma_K)\simeq \cc(\sigma)\otimes_F K\simeq \cc(\sigma)\times \overline{\cc(\sigma)},\]
where $\overline{\cc(\sigma)}$ denotes the conjugate algebra, and similarly for $\cc(\tau_K)$.
Hence, the assumption $c(\sigma_K)=c(\tau_K)\in\br(K)/\qform{[A_K]}$, says that $\cc(\tau)$ is isomorphic either to $\cc(\sigma)$ or to its conjugate, which precisely means that they are isomorphic as $F$-algebras.
\end{proof}

\begin{rem}
As we already mentioned, in the odd degree orthogonal case, Izhboldin has proved a much stronger result, namely that two odd dimensional motivic equivalent quadratic forms are similar. Hence, no assumption on the base field is required in this case.
In some even degree cases, we can slightly weaken the assumption on the base field as follows, by~\cite[Thm.A, Thm.B]{LT} :
\begin{enumerate}
\item For non formally real fields, $I^3(F)=0$ is enough in symplectic and orthogonal types;
\item For formally real fields, $I^3F(\sqrt{-1})=0$ and $F$ SAP is enough for symplectic involutions;
\item For formally real fields, $I^3F(\sqrt{-1})=0$ and $F$ ED is enough for orthogonal involutions on even-degree algebras.
\end{enumerate}
\end{rem}

\begin{rem}
By Springer's theorem, a quadratic form is isotropic over an odd degree extension of the base field if and only if it is isotropic. Equivalently, we have $i_{w,2}(q)=i_w(q)$ for all quadratic form $q$. It is not known whether this also holds for involutions, even over a field satisfying the conditions of the theorem above. Under those conditions, though, we have the following weaker assertion : 
If $(A,\sigma,\tau)$ is an algebra with two involutions of the same type over a field $F$ which is either a non formally real field of cohomological dimension $\leq 2$, or a formally real field with virtual cohomological dimension $\leq 2$ and satisfying the ED property, then 
\[\forall M/F,\ i_{w,2}(\sigma_M)=i_{w,2}(\tau_M)\ \Rightarrow\ \forall M/F,\ i_{w}(\sigma_M)=i_{w}(\tau_M).\]
Indeed, the left condition guarantee that the involutions $\sigma$ and $\tau$ are motivic equivalent by~\cite{dC} (see Lemma~\ref{Charles.lem}). By the above theorem, this implies that the involutions are isomorphic, hence they do have the same Witt index over any extension of the base field. 
\end{rem}

\section{Examples of critical varieties for some exceptional groups}

\setlength{\unitlength}{.5cm}
\newsavebox{\Evi}
\savebox{\Evi}(5,2){\begin{picture}(5,2)
    \multiput(0.5,0.5)(1,0){5}{\circle*{\darkrad}}
    \put(2.5,1.5){\circle*{\darkrad}}

    \put(0.5,0.5){\line(1,0){4}}
    \put(2.5,1.5){\line(0,-1){1}}\end{picture}}

\label{exceptional.section}

The aim of this last section is to study critical varieties for the class of so-called \emph{tractable} semisimple algebraic groups, which includes many groups of exceptional types. As far as classical groups and algebras with involutions are concerned the main homological torsion prime of interest is $p=2$. Yet we now will consider semisimple groups with odd torsion primes and thus recall some of the notations of \cite{dC}. 

Until the end of this section we will only consider for the sake of simplicity semisimple groups of inner type, for which the list of torsion primes is given in \cite{dCG}. As a direct consequence of the Bruhat decomposition, all twisted flag $G$-varieties are critical modulo $p$ if $p$ is not a torsion prime for $G$, or more generally if $G$ is \emph{$p$-split} -which means $G$ splits over a field extension $E/F$ of prime-to-$p$ degree-. Note that as the tables of \cite{dCG} show, the $p$-splitting of exceptional groups often follow a rather basic pattern, motivating the following definition.

\begin{defi}\label{tractable}
The type $T_n$ of an absolutely simple algebraic group is called \emph{$p$-tractable} if the set of possible values of the Tits $p$-indexes of groups of type $T_n$ has cardinality $2$.
\end{defi}

Seemingly restrictive at first glance, the preceding definition covers the following important examples :

\begin{enumerate}
        \item the types $F_4$ and $E_7$ are $3$-tractable;
    \item the types $G_2$ and $E_6$ are $2$-tractable;
    \item the type $E_8$ is $5$-tractable.
\end{enumerate}

\begin{table}[hbt]
\begin{picture}(5,2)
\put(0,0){\usebox{\Evi}}    
    \put(0.5,0.5){\circle{\lrad}}
    \put(4.5,0.5){\circle{\lrad}}
\end{picture}
\hspace{1cm}
\begin{picture}(5,2)
\put(0,0){\usebox{\Evi}}    
\put(2.5,1.5){\circle{\lrad}}
    \multiput(0.5,0.5)(1,0){5}{\circle{\lrad}}\end{picture}
\caption{The $2$-indexes of the $2$-tractable type $^{1\!}E_6$}
\end{table}

\vspace{-0.5cm}

Recall that a twisted flag $G$-variety is called $p$-isotropic if it carries a zero-cycle of prime-to-$p$ degree, and $p$-anisotropic otherwise. Following our conventions $X_{\Theta,G}$ is $p$-isotropic if and only if $\Theta$ is circled in the Tits $p$-index of $G$. Reformulating definition \ref{tractable}, one may say that $G$ is $p$-tractable if $G$ is not $p$-split and any $p$-anisotropic twisted flag $G$-variety is generically split, in the sense of \cite{PSZ}. Using this, we prove : 

\begin{thm}\label{zzz}
Assume that $G$ is a $p$-tractable semisimple group of inner type. Any $p$-anisotropic twisted flag $G$-variety $X_{\Theta,G}$ is critical for $G$ modulo $p$.
\end{thm}

\begin{proof}
Let $G$ and $G'$ be inner twisted forms of the same quasisplit group, and assume their common type is $p$-tractable. 
Let $\Theta\subset\Delta$ be such that $X_{\Theta,G}$ is $p$-anisotropic, that is there is at least one element in $\Theta$ which is not circled in the Tits $p$-index of $G$. 
Assume that $M(X_{\Theta,G};p)$ and $M(X_{\Theta,G'};p)$ are isomorphic. 
Then, in particular, $G'$ is non-split, so
$G'$ and $G$ have the same Tits $p$-index and $X_{\Theta, G'}$ also is $p$-anisotropic. 
Consider an extension $K$ of the base field $F$. If both motives are non-split over $K$, then the groups are non split, hence have the same Tits $p$-index as over the base field, since there is only one non split possible value. Otherwise, both motives and both groups are split over $K$. Therefore, $G$ and $G'$ have the same higher Tits $p$-indexes.  By~\cite[Thm 15]{dC}, this proves that $G$ and $G'$ are motivic equivalent, hence $X_{\Theta, G}$ is critical. 
\end{proof}

\end{document}